\documentclass[a4paper]{amsart}

\usepackage[utf8]{inputenc}
\usepackage{mathtools}
\usepackage{amssymb,amsfonts,amsmath}
\usepackage{enumitem}
\usepackage{color}
\usepackage{mathrsfs}
\usepackage{graphicx}
\usepackage{verbatim}
\usepackage{tikz-cd}
\usepackage{tikz}
\usepackage{url} 
\usepackage{mathabx}
\usepackage[dvipsnames]{xcolor}
\usepackage{cleveref}

\numberwithin{equation}{section}
\newtheorem{theorem}{Theorem}[section]
\newtheorem{corollary}[theorem]{Corollary}
\newtheorem{lemma}[theorem]{Lemma}

\newtheorem{proposition}[theorem]{Proposition}
\newtheorem{question}[theorem]{Question}

\theoremstyle{definition}
\newtheorem{definition}[theorem]{Definition}

\theoremstyle{remark}
\newtheorem{remark}[theorem]{Remark}

\newcommand{\C}{\mathbb{C}}
\newcommand{\N}{\mathbb{N}}
\newcommand{\Q}{\mathbb{Q}}
\newcommand{\R}{\mathbb{R}}
\newcommand{\Z}{\mathbb{Z}}

\newcommand{\TT}{\mathcal{T}}
\newcommand{\CC}{\mathfrak{C}}
\newcommand{\XX}{\mathfrak{X}}
\newcommand{\lkd}{{\rm lk}^{\downarrow}}

\DeclareMathOperator{\AAut}{AAut}
\DeclareMathOperator{\Aff}{Aff}

\DeclareMathOperator{\Area}{Area}
\DeclareMathOperator{\Aut}{Aut}
\DeclareMathOperator{\BS}{BS}
\DeclareMathOperator{\BB}{BB}

\DeclareMathOperator{\diag}{diag}

\DeclareMathOperator{\F}{F}

\DeclareMathOperator{\FP}{FP}

\DeclareMathOperator{\Flag}{Flag}

\DeclareMathOperator{\GL}{GL}
\DeclareMathOperator{\id}{id}

\DeclareMathOperator{\HNN}{HNN}

\DeclareMathOperator{\Homeo}{Homeo}

\DeclareMathOperator{\SL}{SL}

\newcommand{\defeq}{\mathrel{\mathop{:}}=}

\renewcommand{\epsilon}{\varepsilon}

\title[Simple groups separated by
homological finiteness properties]{Infinitely presented simple groups separated by
homological finiteness properties}

\author[C. Llosa Isenrich]{Claudio Llosa Isenrich}
\address{Karlsruhe Institute of Technology, Englerstr.\ 2, 76131 Karlsruhe, Germany}
\email{claudio.llosa@kit.edu}

\author[E. Schesler]{Eduard Schesler}
\address{Karlsruhe Institute of Technology, Englerstr.\ 2, 76131 Karlsruhe, Germany}
\email{eduardschesler@googlemail.com}

\author[X. Wu]{Xiaolei Wu}
\address{Shanghai Center for Mathematical Sciences, Jiangwan Campus, Fudan University, No.2005 Songhu Road, Shanghai, 200438, P.R. China}
\email{xiaoleiwu@fudan.edu.cn}

\subjclass[2010]{Primary 20E32; Secondary 20E08}

\keywords{Simple groups, finiteness properties}

\begin{document}
\begin{abstract}
Given a finitely generated linear group $G$ over $\Q$, we construct a simple group $\Gamma$ that has the same finiteness properties as $G$ and admits $G$ as a quasi-retract.
As an application, we construct a simple group of type $\FP_{\infty}$ that is not finitely presented.
Moreover we show that for every $n \in \N$ there is a simple group of type $\FP_n$ that is neither finitely presented nor of type $\FP_{n+1}$.
Since our simple groups arise as R\"over--Nekrashevych groups, this answers a question of Zaremsky.
\end{abstract}
\maketitle

\section{Introduction}

A group is of type $\F_n$ (resp.\ $\F$) if it admits a classifying space with finite $n$-skeleton (resp.\ finite skeleton).
These so-called (homotopical) finiteness properties were introduced by Wall~\cite{Wall65} and generalize the properties of being finitely generated, which is equivalent to being of type $\F_1$, and being finitely presented, which is equivalent to being of type $\F_2$.
Since then, finiteness properties are typically among the first properties that one aims to determine when one is interested in a group.
Apart from a few notable exceptions~\cite{BuxKoehlWitzel13}, it is usually the case that a finitely presented group that arises in nature, is of type $\F_{\infty}$, i.e.\ of type $\F_n$ for every $n$.  Thus, it is common to say that groups which are not $\F_{\infty}$ have \emph{exotic finiteness properties}. The first examples of finitely presented groups of type $\F_n$ and not $\F_{n+1}$ were constructed by Stallings for $n=2$ \cite{Stallings63} and Bieri for all $n\geq 3$ \cite{Bieri1976}. Since then examples of groups of type $\F_n$ and not $\F_{n+1}$ have been produced with various additional properties, see e.g.~\cite{SkipperWitzelZaremsky19,Schesler23Sigma,LlosaIsenrichPy24} for some recent instances of this.
Even more exotic finiteness properties can be obtained by considering a homological analogue of the properties $\F_n$.
These were introduced by Bieri~\cite{Bieri81}, who defined a group to be of type $\FP_n(R)$ (resp.\ $\FP(R)$) for a commutative ring $R$, if the trivial $RG$-module $R$ admits a resolution by projective $RG$-modules that is finitely generated up to degree $n$ (resp.\ finitely generated in all degrees and of finite length).
It can be easily verified that a group of type $\F_n$ is also of type $\FP_n(R)$ for every $R$ and that the properties $\F_1$ and $\FP_1(R)$ coincide as soon as $R \neq 0$.
However, the question whether $\FP_n(R)$ implies $\F_n$ for $R \neq 0$ remained open until the seminal work of Bestvina and Brady~\cite{BestvinaBrady97}, in which they constructed a group of type $\FP(\Z)$ that is not of type $\F_2$.
Moreover they showed that for every $n \in \N$ there exists a group of type $\FP_n(\Z)$ that is neither of type $\FP_{n+1}(\Z)$ nor of type $\F_2$.
Recall that the restriction to groups that are not of type $\F_2$ makes the situation more interesting here since for groups of type $\F_2$ the properties $\F_n$ and $\FP_n(\Z)$ coincide.
Until today, the known sources of infinitely presented groups of type $\FP_2(\Z)$ are very limited, see e.g.~\cite{KrophollerLearySoroko20,Kropholler22}, and there is barely a naturally occurring class of groups known that contains such groups.
The following result in particular tells us that simple groups are among these classes and that in fact simple groups admit the same variety of finiteness properties as the groups of Bestvina and Brady.
To make this result more precise, we recall that the flag complex $\Flag(\Gamma)$ of a graph $\Gamma$ is the simplicial complex, with $k$-simplices the $(k+1)$-cliques in $\Gamma$.
We will further denote by $V(\Gamma)$ and $E(\Gamma)$ the vertex and edge sets of $\Gamma$.

\begin{theorem}\label{mainthm:BB-type-simple-groups}
For every finite graph $\Gamma$ there is a simple group $G_{\Gamma}$ with the following finiteness properties, where $R\neq 0$ is any commutative unital ring:
    \begin{enumerate}
        \item $G_{\Gamma}$ is of type $\F_n$ if and only if $\Flag(\Gamma)$ is $(n-1)$-connected;
        \item $G_{\Gamma}$ is of type $\FP_n(R)$ if and only if $\Flag(\Gamma)$ is homologically $(n-1)$-connected over $R$;
        \item $G_{\Gamma}$ is $\F_\infty$ if and only if $\Flag(\Gamma)$ is contractible;
        \item $G_{\Gamma}$ is $\FP_\infty(R)$ if and only if $\Flag(\Gamma)$ is $R$-acyclic.
    \end{enumerate}
\end{theorem}

It is well-known that for every $n\geq 1$ one can choose $\Gamma$ so that $\Flag(\Gamma)$ is homologically $(n-1)$-connected over $\mathbb{Z}$, but not homologically $n$-connected over $\mathbb{Z}$ and not simply connected.
There are also examples of finite acyclic complexes that are not simply connected \cite{BestvinaBrady97}. Thus, we obtain the following direct consequences of \Cref{mainthm:BB-type-simple-groups}.

\begin{corollary}\label{introcor:BB-type-simple-groups-n-not-n+1}
For every $n \in \N$ there is a simple group of type $\FP_{n}(\mathbb{Z})$ that is neither of type $\FP_{n+1}(\mathbb{Z})$ nor finitely presented.
\end{corollary}

\begin{corollary}\label{introcor:BB-type-simple-groups-infty}
There is a simple group of type $\FP_\infty(\mathbb{Z})$ that is not finitely presented.
\end{corollary}

Independently and with different methods Bonn and Giersbach recently constructed non-discrete simple totally disconnected locally compact groups which are $\FP_2(\Z)$, but not compactly presented~\cite{BonnGiersbach25}.

In fact, we are neither aware of a way to transfer their methods to the discrete case, nor of a way to apply the methods in this paper to the non-discrete case.

\medskip

As the formulation of Theorem~\ref{mainthm:BB-type-simple-groups} already suggests, its proof is based on techniques that allow us to transfer the finiteness properties of the Bestvina--Brady group $\BB_{\Gamma}$ to the simple group $G_{\Gamma}$.
The following result, from which Theorem~\ref{mainthm:BB-type-simple-groups} will be deduced, sheds a bit more light on the role that is played by the group $\BB_{\Gamma}$ for the proof of Theorem~\ref{mainthm:BB-type-simple-groups}.

\begin{theorem}\label{introthm:linear-groups-as-retracts}
Let $H$ be a finitely generated subgroup of $\GL_n(\Q)$ for some $n \in \N$.
There exists a simple group $G$ that has the following properties:
\begin{enumerate}
\item $G$ has the same finiteness properties as $H$,
\item $H$ is a subgroup of $G$,
\item $G$ admits a quasi-retract onto $H$.
\end{enumerate}
\end{theorem}

\begin{remark}\label{introrem:algebraic-closure}
Recall that every finitely generated subgroup of $\GL_n(\overline{\Q})$, where $\overline{\Q}$ denotes the algebraic closure of $\Q$, can be embedded into $\GL_m(\Q)$ for an appropriate number $m \in \N$.
As a consequence, one could replace $\Q$ with $\overline{\Q}$ in Theorem~\ref{introthm:linear-groups-as-retracts}.
\end{remark}

Since every Bestvina--Brady group $\BB_{\Gamma}$, being a subgroup of a right-angled Artin group, is a subgroup of a right-angled Coxeter group~\cite{DavisJanuszkiewicz00}, it follows that $\BB_{\Gamma}$ is linear over $\mathbb{\Z}$ as the cosine matrix of a right-angled Coxeter group is integral.
Thus we can deduce Theorem~\ref{mainthm:BB-type-simple-groups} from Theorem~\ref{introthm:linear-groups-as-retracts} by considering the case where $H = \BB_{\Gamma}$, see Section~\ref{sec:applications}.
The proof of Theorem~\ref{introthm:linear-groups-as-retracts} is based on the R\"over--Nekrashevych construction (see Definition~\ref{def:roever-nekrashevich-groups}), which gained great popularity in the construction of finitely presented simple groups, especially in the context of the Boone--Higman conjecture, see~\cite{BelkBleakMatucciZaremsky23} for some background.
This construction takes as input a self-similar subgroup $H$ of the automorphism group of a $d$-adic rooted tree $T$ and associates to it a subgroup $V_d(H)$ of the homeomorphism group of the boundary of $T$, which is known as the R\"over--Nekrashevych group associated to $H$.
In order to realize the group $G$ from Theorem~\ref{introthm:linear-groups-as-retracts} as a R\"over-Nekrashevych group $V_d(H)$, there are three difficulties to overcome.
The first one is that $V_d(H)$ is not always simple.
In fact $V_d(H)$ can have infinite abelianization.
The second difficulty is that $H$ does not need to be a quasi-retract of $V_d(H)$.
This can be deduced from the structure of the R\"over--Nekrashevych group associated to Grigorchuk's group, which was shown to be finitely presented by R\"over~\cite{Roever99}.
Indeed, since Grigorchuk's group is not finitely presented, it cannot arise as a quasi-retract of a finitely presented group~\cite{Alonso94}.
The third and most dramatic difficulty is that there is no reason to believe that the group $H$ in Theorem~\ref{introthm:linear-groups-as-retracts} is self-similar, so that we cannot consider its corresponding R\"over-Nekrashevych group.
In fact, the question of the existence of self-similar structures in linear groups is not well understood and deserves further study, see~\cite{Kapovich12} for some results in this direction.
The fruitfulness of such a study was proven, for example, by Skipper, Witzel and Zaremsky~\cite{SkipperWitzelZaremsky19}, who found self-similar representations for a carefully chosen family $(H_n)_{n \in \N}$ of linear groups over global function fields, such that the corresponding R\"over-Nekrashevych groups $V_d(H_n)$ are of type $\F_n$ but not of type $\F_{n+1}$.
Moreover they showed that their groups $V_d(H_n)$ admit $H_n$ as a quasi-retract.
To overcome the first of the aforementioned difficulties of constructing such R\"over--Nekrashevych groups, they introduced the notion of coarse diagonality of a self-similar group $H$, which means that for every $h \in H$ and every level-$1$ state $h_i$ of $h$, the element $h^{-1}h_i$ has finite order, and showed that the R\"over--Nekrashevych group of a finitely generated coarsely diagonal self-similar group is virtually simple.
This rather restrictive property was relaxed by Zaremsky~\cite{Zaremsky25} to the notion of weak diagonality, which is still sufficient to deduce the virtual simplicity of $V_d(H)$, while only requiring that there is a generating set $S$ of $H$ such that, for all $s \in S$ and every level-$1$ state $s_i$ of $s$, the element $s^{-1}s_i$ has finite order in the abelianization of $H$.
To overcome the second of the aforementioned difficulties, Skipper, Witzel and Zaremsky~\cite{SkipperWitzelZaremsky19} introduced the notion of persistency of a self-similar group and combined it with the well-known notion of rationality, which allowed them to deduce the existence of a quasi-retract onto their linear groups $H_n$.
Finally, they twisted some natural self-similar action of $H_n$ that satisfied all of the required properties to deduce their result.
The first goal of this paper will be to show that a certain class of linear groups $H$ over $\Q$ admits self-similar representations for which one can completely abandon the assumptions of rationality, weak diagonality, and persistence, while still retaining that the corresponding R\"over--Nekrashevych group $V_d(H)$ is virtually simple, has the same finiteness properties as $H$, and admits $H$ as a quasi-retract, see Sections~\ref{sec:quasi-retracts}, ~\ref{sec:finite-abel},~\ref{sec:realizing-persistent-retracts}.
In the second step we produce a self-similar extension of an arbitrary finitely generated subgroup of $\GL_n(\Q)$, which ultimately allows us to get rid of the self-similarity assumption itself and thereby prove Theorem~\ref{introthm:linear-groups-as-retracts}, see Subsection~\ref{subsec:self-similar-extensions}.
Since after all the group $G$ in Theorem~\ref{introthm:linear-groups-as-retracts}, and hence the group in Corollary~\ref{introcor:BB-type-simple-groups-n-not-n+1}, is a R\"over--Nekrashevych group, and therefore Thompson-like in the sense of Skipper and Zaremsky~\cite[Definition 3.2]{SkipperZaremsky24}, we obtain an affirmative answer to the following question of Zaremsky~\cite[Question 113]{BurilloBuxNucinkis18}.

\begin{question}[Zaremsky] 
In the family of Thompson-like groups, can one find non-finitely presentable groups of type $\FP_2(\Z)$?
\end{question}

As a further application of Theorem~\ref{introthm:linear-groups-as-retracts}, we obtain that every coarse invariant of a finitely generated subgroup of $\GL_n(\Q)$ arises as a lower bound for the invariant of a simple group with the same finiteness properties.
An example of such a coarse invariant is provided by (higher dimensional) Dehn functions, which allows us to prove the following, see Corollary~\ref{cor:higher-dehn-fct-general}.

\begin{corollary}\label{introcor:higher-dehn-fct-general}
Every (high-dimensional) Dehn function of a subgroup of $\GL_n(\Q)$ arises as a lower bound on the corresponding Dehn function of a simple group with the same finiteness properties.
\end{corollary}

Corollary~\ref{introcor:higher-dehn-fct-general} provides us with a tool that is potentially capable of producing finitely presented simple groups with very fast growing Dehn functions.
The question on the existence of such groups naturally arises in the light of the fact that finitely presented simple groups have solvable word problem and therefore a computable Dehn function so that it is natural to ask how fast it can be, see e.g.~\cite{ZaremskyMO} for a discussion on this initiated by Zaremsky.
However, despite of the fact that no recursive upper bound on Dehn functions of finitely presented linear groups is known so far,
though conjectured to exist by Gersten and Riley~\cite[Conjecture 8.7]{GerstenRiley05}, the fastest known Dehn function of a linear group is exponential.
As exponential functions also arise as high-dimensional Dehn functions of linear groups over $\Q$, we obtain the following consequence of Corollary~\ref{introcor:higher-dehn-fct-general}.

\begin{corollary}\label{introcor:higher-dehn-fct}
For every finite subset $I \subseteq \N$, there is a simple group of type $\F_{\infty}$ whose $n$-dimensional Dehn function is at least exponential for every $n \in I$.
\end{corollary}

A special case of Corollary~\ref{introcor:higher-dehn-fct} was obtained by Zaremsky~\cite{Zaremsky25}, who carefully calibrated certain actions of the Baumslag-Solitar groups $\BS(1,n) \leq \GL_2(\Q)$ on an appropriate rooted tree to make them rational, weakly diagonal, and persistent. This allowed him to obtain $\BS(1,n)$ as a quasi-retract of its R\"over--Nekrashevych group and therefore to deduce the existence of finitely presented simple groups whose Dehn function is bounded from below by an exponential function.

As another consequence of our considerations, we obtain the first known examples of self-similar groups that are separated by homological finiteness properties, see Proposition~\ref{prop:BB-type-selfsim-groups}.

\begin{proposition}\label{introprop:BB-type-selfsim-groups}
For every $n \in \N$ there is a self-similar group of type $\FP_{n}(\mathbb{Z})$ that is neither of type $\FP_{n+1}(\mathbb{Z})$ nor finitely presented.
Moreover there is a self-similar group of type $\FP_{\infty}(\mathbb{Z})$ that is not finitely presented.
\end{proposition}

Finally, it might be worth mentioning that Theorem~\ref{introthm:linear-groups-as-retracts} seems to push the current techniques for constructing R\"over--Nekrashevych groups with interesting finiteness properties via linear groups, which began in Scott~\cite{Scott84a} and evolved e.g.\ in~\cite{SkipperWitzelZaremsky19,Zaremsky25embedding,Zaremsky25}, to their limits.
As a natural next step in this direction, we propose the following question.

\begin{question}\label{quest:linear-groups-as-retracts}
Let $H$ be a finitely generated subgroup of $\GL_n(\C)$ for some $n \in \N$.
Does there exist a simple group $G$ that has the following properties?
\begin{enumerate}
\item $G$ has the same finiteness properties as $H$,
\item $H$ is a subgroup of $G$,
\item $G$ admits a quasi-retract onto $H$.
\end{enumerate}
\end{question}

As indicated above, an answer to this question will likely require completely new techniques.
We take the opportunity to formulate our second question that naturally arises within the framework of R\"over--Nekrashevych groups.

\begin{question}\label{quest:embedding}
Does every finitely generated linear group embed into a finitely generated self-similar group?
\end{question}

We conclude with the following general question. There is no a priori reason why simple groups of type $\FP_2(\Z)$ should have solvable word problem. So in the context of the first construction of infinitely presented simple groups of type $\FP_2(\Z)$ and the recent interest in the Boone--Higman Conjecture \cite{BelkBleakMatucciZaremsky23}, it is natural to ask:
\begin{question}\label{ques:hml-BH-conj}
Does every countable group embed in a simple group of type $\FP_2(\Z)$?
\end{question}
Recall that Leary proved in \cite{Leary2018} that every countable group embeds in a $\FP_2(\Z)$ group. On the other hand, it is a classical theorem of Hall \cite{Hall1974} that every countable group embeds in a finitely generated simple group, see also \cite{Gorjuvskin74,Schupp76,BelkZaremsky2022} for some different proofs.

\subsection*{Acknowledgments.}
CLI was partially supported by the DFG projects 281869850 and 541703614. XW is a member of LMNS and is supported by NSFC No.12326601. He thanks Fan Wu for an inspiring discussion that leads to Question \ref{ques:hml-BH-conj}. The authors are grateful to Roman Sauer for a number of helpful discussions. 

\section{Background on R\"over--Nekrashevych groups}

We introduce R\"over--Nekrashevych groups, following the exposition in \cite{SkipperWitzelZaremsky19} and \cite{BelkMatucci2024}. For a more detailed introduction to the subject, we refer to \cite{Nekrashevych05, Nekrashevych18-2}.

\subsection{The R\"over--Nekrashevych group $V_d(G)$}

For $d \in \N$, let $\TT_d$ be the infinite rooted $d$-ary tree whose vertex set $V(\TT_d)$ is given by the set $X_d^{\ast}$ of words over the alphabet $X_d=\left\{1,\ldots, d\right\}$ and whose (directed) edges are given by the pairs $(u,ux)$ with $u\in V(\TT_d)$ and $x \in X_d$.
In other words, $\TT_d$ is the Cayley graph of the free monoid $X_d^{\ast}$ with respect to $X_d$.
For $f\in \Aut(\TT_d)$ and $u\in V(\TT_d)$ there is a unique $f^u\in \Aut(\TT_d)$ with $f(uv)= f(u) f^u(v)$ for every $v\in X_d^{\ast}$. We call $f^u$ the \emph{state} of $f$ at $u$.
Note that $f$ fixes the root of $\TT_d$.
As a consequence, $f$ restricts to an action on the \emph{level}-$n$ vertices of $\TT_d$, which are represented by words of length $n$ over $X_d$.
In the case where $n=1$, this provides us with a permutation $\rho(f) \in S_d$ in the symmetric group $S_d$ on $d$ elements via $\rho(f)(x)=f(x)$.
If no ambiguity is possible, we will also write $\rho(f)$ to denote the element of $\Aut(\TT_d)$ that is given by $\rho(f)(xv)=f(x)v$ for every $x \in X_d$ and every $v \in X_d^{\ast}$.
Automorphisms of the latter type are called \emph{rooted} and can be combined with the states of $f$ to decompose $f$ via the so-called \emph{wreath recursion}, which is given by the isomorphism
\[
\Aut(\TT_d)\to S_d\wr \Aut(\TT_d),\ f\mapsto \rho(f)(f_1,\ldots,f_d),
\]
where $(f_1,\ldots,f_d)$ acts on $\TT_d$ by mapping a word of the form $xv$ to $xf_x(v)$.
Using the wreath recursion, we can define a subset $S \subseteq \Aut(\TT_d)$ to be \emph{self-similar} if for every $s\in S$ and $u\in V(\TT_d)$ we have $s^u\in S$.
A subgroup $G \leq \Aut(\TT_d)$ is called \emph{self-similar} if it is self-similar as a subset.
If $G$ admits a finite generating set $S$ that is self-similar, then $G$ is said to be an \emph{automata group}.
We say that an element $f\in \Aut(\TT_d)$ is \emph{finite state}, if its set of states is finite and that a subgroup $G\leq \Aut(\TT_d)$ is \emph{finite state}, if all its elements are finite state.
Note that a subgroup $G \leq \Aut(\TT_d)$ is an automata group if and only if it is finite state and self-similar.
When dealing with groups that act on rooted trees, it is often useful to consider for each $u\in V(\TT_d)$ the embedding $\iota_u \colon \Aut(\TT_d) \to \Aut(\TT_d)$ which maps an element $f \in \Aut(\TT_d)$ to the automorphism that acts by $\iota_u(uw)=uf(w)$ on words with prefix $u$ and trivially on all other words. 
The groups we are interested in act on the boundary $\partial \TT_d=X_d^{\omega}$ of $\TT_d$, which is the $d$-ary Cantor set $\CC_d=X_d^{\omega}$ of infinite words in $X_d$. For $u\in X^{\ast}$, we will denote by
\[
\CC_d(u) \defeq \left\{uv\mid v\in X_d^{\omega}\right\} \subseteq \CC_d
\]
the cone of elements of $\CC_d$ starting with $u$.
Note that $\Aut(\TT_d)$ embeds in $\Aut(\partial \TT_d)$ via its natural action on the set of finite prefixes of an element $u\in X_d^{\omega}$.
Regarding this, we will sometimes view $\Aut(\TT_d)$ as a subgroup of $\Homeo(\partial \TT_d)$ via this embedding.
A finite subtree $T$ of $\TT_d$ will be called \emph{complete} if it is rooted, i.e.\ contains the root $\emptyset$ of $\TT_d$ and has the property that whenever we have $u, ux \in V(T)$ for some $x\in X_d$, then we also have $uy \in V(T)$ for every $y\in X_d$. A vertex of a finite complete subtree $T$ is called a \emph{leaf}, if it has valence $1$. Note that by the choice of the alphabet $X_d$, there is a natural lexicographic order on the leaves of a finite complete subtree of $\TT_d$.

\begin{definition}\label{def:almost-automorphisms}
Let $T_{-}$ and $T_{+}$ be complete rooted subtrees of $\TT_d$ with the same number of leaves. Denote by $u_1, \ldots, u_n$ the ordered leaves of $T_{-}$ and by $v_1,\ldots, v_n$ the ordered leaves of $T_{+}$. Let $f_1,\ldots, f_n \in \Aut(\TT_d)$ and let $\sigma \in S_n$ be a permutation.
Then we write
\[
\left[T_{-}, \sigma(f_1,\ldots, f_n),T_+\right]
\]
to denote the element in $\Homeo(\partial \TT_d)$ that maps $\CC_d(v_i)$ to $\CC_d(u_{\sigma(i)})$ via $v_iw\mapsto u_{\sigma(i)}f_i(w)$.
Elements of $\Homeo(\partial \TT_d)$ that arise in that way will be called \emph{almost automorphisms} of $\TT_d$.
The group of all almost automorphisms of $\TT_d$ will be denoted by $\AAut(\TT_d)$.
\end{definition}

\begin{definition}\label{def:roever-nekrashevich-groups}
    Let $G\leq \Aut(\TT_d)$ be a self-similar group. The \emph{R\"over--Nekrashevych group} associated to $G$ is the subgroup of $\AAut(\TT_d)$ that is given by
    \[
        V_d(G):= \left\{\left[T_-,\sigma(f_1,\ldots,f_n),T_+\right] \in \AAut(\TT_d) \mid f_1,\ldots, f_n\in G\right\}.
    \]
\end{definition}

Note that while as a set $V_d(G)$ makes sense even if $G$ is not self-similar, the self-similarity condition guarantees that this is actually a subgroup. For $G=\left\{1\right\}$ we obtain the Higman--Thompson group $V_d(G)=V_d$, which is well-known to be of type $F_{\infty}$.
Let $C$ be the unique complete subtree of $\TT_d$ with one root and $d$ leaves.
Using $C$, we define an embedding of a self-similar group $G$ into its R\"over--Nekrashevych group by
\[
\iota_1 \colon G \to V_d(G),\ g \mapsto (C, \mathrm{id}(g,1,\cdots, 1),C).
\]
We collect some well-known properties of R\"over--Nekrashevych groups.

\begin{proposition}\label{prop:prpty-Vd(G)}
Let $G\leq \Aut(\TT_d)$ be self-similar. 
Then the following hold:
    \begin{enumerate}
        \item $V_d(G)=\left\langle \iota_1(G), V_d\right\rangle$ and in particular $V_d(G)$ is finitely generated if $G$ is finitely generated;
        \item the commutator subgroup $[V_d(G),V_d(G)]$ is simple;
    \end{enumerate}
\end{proposition}
\begin{proof}
    For (1) see \cite[Lemma 5.11]{Nekrashevych18-2} and (2) is \cite[Theorem 4.7]{Nekrashevych18-2}.
\end{proof}

\begin{proposition}[{\cite[Theorem 4.8]{Nekrashevych18-2}}]\label{prop:abelianization-Vd(G)}
    Let $G\leq \Aut(\TT_d)$ be self-similar and let $\pi:G\to G_{ab}=G/[G,G]$ be the abelianization morphism. Then the abelianization $V_d(G)_{ab}$ of $V_d(G)$ can be described as follows:
    \begin{enumerate}

        \item If $d$ is even, then $V_d(G)_{ab}$ is isomorphic to the quotient of $G_{ab}$ by the relations $\pi(g)=\sum_{i=1}^d \pi(g_i)$ for all $g=\rho(g)(g_1,\ldots,g_d)\in G$.
        \item If $d$ is odd, then $V_d(G)_{ab}$ is isomorphic to the quotient of $\Z/2\Z \oplus G_{ab}$ by the relations $\pi(g)={\rm sign}(\rho(g)) \oplus \sum_{i=1}^d \pi(g_i)$ for all $g=\rho(g)(g_1,\ldots,g_d)\in G$.
  \end{enumerate}
\end{proposition}
\begin{remark}
    Since the relations in \Cref{prop:abelianization-Vd(G)} are compatible with the group operation, to determine $V_d(G)_{ab}$ it suffices to quotient $G_{ab}$ (resp.\ $\Z/2\Z\oplus G_{ab}$) by the relations induced by a generating set for $G$.
\end{remark}

\subsection{Expansions of triples}

Different triples $(T_-, \sigma(f_1,\ldots, f_n), T_+)$ can define the same element $[T_-, \sigma(f_1,\ldots, f_n), T_+]\in\AAut(\TT_d)$. To give a precise criterion for when this is the case, we introduce expansions, following \cite{Zaremsky21} and also \cite{WuWuZhaoZhou25}.
Expansions will play an important role in the construction of a Stein--Farley complex the R\"over--Nekrashevych groups act on and thus in the proof of their finiteness properties. For the proof of finiteness properties we will need to consider the more general situation of triples $(F_-,\sigma(f_1,\ldots,f_n),F_+)$, where $F_-$ and $F_+$ are $d$-ary rooted complete forests with $n$ leaves, that is, ordered sequences of finitely many finite $d$-ary rooted complete trees. We will write $1_n$ to denote the \emph{trivial forest with $n$ roots}. We will now assume that all forests and trees are finite, rooted and complete.
A ($d$-ary) \emph{caret} is a tree with $d$ leaves.
Let $F'$ and $F$ be forests.
We call $F'$ a \emph{simple expansion} of $F$ if $F'$ is obtained from $F$ by attaching a single caret along its root to one of the leaves of $F$.
We call $F'$ an \emph{expansion} of $F$ if it is obtained from $F$ by a sequence of simple expansions.
An expansion $F'$ of $F$ will be called \emph{elementary} if $F'$ is obtained from $F$ by a sequence of simple expansions of leaves of $F$.
Similarly, a forest $F$ is called \emph{elementary} (resp.\ \emph{simple}) if it is an elementary (resp.\ simple) expansion of a trivial forest.
Let $F_J^{(n)}$ denote the elementary $n$-rooted forest with $|J|$ carets whose roots correspond to the elements of $J\subseteq \left\{1,\ldots,n\right\}$.
With these notions, we can define a \emph{simple expansion} of a triple $(F_-,\sigma(f_1,\ldots,f_n),F_+)$ as a triple of the form $$(F'_-,\sigma'(f_1,\ldots,f_{i-1},f_i^{1},\ldots, f_i^{d} 
,f_{i+1},\ldots,f_n),F'_+)$$
for some $i$, where $F'_+$ (resp.\ $F'_-$) is a simple expansions of $F_+$ (resp.\ $F_-$) in the $i$-th (resp.\ $\sigma(i)$-th) leaf, $\rho(f_i)(f_i^1,\ldots,f_i^d)$ is the wreath recursion of $f_i$, and $\sigma'\in S_{n+d-1}$ is the permutation induced by $\sigma$ and $\rho(f_i)$.\footnote{To give a precise definition of $\sigma'$, label the ordered leaves of $T_+'$ by $(1,\ldots, i-1,i1, \ldots, id, i+1,\ldots, n)$ and the ordered leaves of $T_-'$ by $(1,\ldots, \sigma(i)-1,\sigma(i)1,\ldots, \sigma(i)d,\sigma(i)+1,\ldots, n)$.
Then $\sigma'(j)=\sigma(j)$ for $j\neq i$ and $\sigma'(ik)=\sigma(i)(\rho(f_i)(k))$ for $j=i$.} As above we can define \emph{expansions} and \emph{elementary expansions} of triples as sequences of simple expansions. We will further call a triple $(F_-,\sigma(f_1,\ldots,f_n),F_+)$ a \emph{contraction} of a triple $(F'_-,\sigma'(f'_1,\ldots,f'_{n'}),F'_+)$ if $(F'_-,\sigma'(f'_1,\ldots,f'_{n'}),F'_+)$ is an expansion of $(F_-,\sigma(f_1,\ldots,f_n),F_+)$. We can now make precise when two triples represent the same element of $\AAut(\TT_d)$.

\begin{lemma}[{\cite[Lemma 2.11]{SkipperWitzelZaremsky19} and also \cite[Lemma 2.3]{LeBoudec2017}}]
    Two triples $[T_-,\sigma(f_1,\ldots,f_n),T_+]$ and $[U_-,\nu(g_1,\ldots,g_n),U_+]$ represent the same element of $\AAut(\TT_d)$ if and only if they have a common expansion.
\end{lemma}

Motivated by the latter result, we will denote the equivalence class of a triple $(F_-,\sigma(f_1,\ldots,f_n),F_+)$ with respect to sequences of expansions and contractions by $[F_-,\sigma(f_1,\ldots,f_n),F_+]$, which we will often shorten by writing $[F_-, (\sigma,f),F_+]$.
An advantage of expansions is that given two equivalence classes $[F_-,(\sigma,f),F_+]$ and $[U_-,(\nu,g),U_+]$ such that $F_+$ and $U_-$ have the same number of roots, there are expansions $[F'_-,(\sigma',f'),F_+']$ and $[U'_-,(\nu',g'),U_+']$ with $F_+'=U'_-$. This allows us to introduce a groupoid structure on the set $\mathcal{F}_d$ of equivalence classes of triples by defining the product
\[
    [F_-,(\sigma,f),F_+] [U_-,(\nu,g),U_+]:=[F'_-,(\sigma',f')\circ(\nu',g'), U_+'],
\]
which restricts to the group structure on $\AAut(\TT_d)$ when restricting to triples where both forests are trees. Restricting to triples $[F_-,(\sigma,f),F_+]$ such that the states of $f$ are in $G$ defines a subgroupoid $\mathcal{F}_d(G)$. Then $\mathcal{F}_d$ is the special case when $G=\Aut(\TT_d)$.
We finally have all ingredients to define the subsets
\[
    \mathcal{F}_{d,n,m}(G):=\left\{[F_-,(\sigma,f),F_+]\in \mathcal{F}_d\mid  \mbox{ $F_-$ $n$-rooted and $F_+$ $m$-rooted}\right\},
\]
which we will use in the next section to prove positive results on the finiteness properties of $V_d(G)$.

\section{Finiteness properties of R\"over--Nekrashevych groups}

In this section, we will run the standard Stein--Farley complex argument to prove that the homological and homotopical finiteness properties of $V_d(G)$ are at least as good as those of $G$. To our knowledge this has only been done for the homotopical finiteness properties $F_n$ in the literature. However, the proof adapts readily to the homological case by applying the homological version of Brown's criterion. Readers familiar with this approach can skip straight to Section \ref{sec:quasi-retracts}.

\subsection{A poset structure on $\mathcal{F}_d(G)$}

A \emph{split} is an element of $\mathcal{F}_d(G)$ of the form $[F,(\id,\id),1]$, where $1=1_n$ for some $n\in \N$ and $\id$ denotes the identity in $S_n$, respectively $n$ copies of the identity in $\Aut(\TT_d)$. We can equip $\mathcal{F}_d(G)$ with a poset structure by defining $[U_-,(\nu,g),U_+]\leq  [U_-,(\nu,g),U_+] [F,(\id,\id),1]$ for any split such that the product on the right is well-defined. It is an easy and standard check that this is indeed a well-defined poset structure and that we have equality if and only if $F=1_n$.
We call a split \emph{elementary} (resp.\ \emph{simple})  if $F$ is elementary (resp.\ simple). For $x,y\in \mathcal{F}_d(G)$, we write $x\preccurlyeq y$ if $y=xu$ for an elementary split $u$.
Note that we obtain an order preserving left action of $V_d(G)$ on the subset of the poset $\mathcal{F}_d(G)$ defined by
\[
    \mathcal{F}_{d,1}(G):=\bigsqcup_{n\geq 1} \mathcal{F}_{d,1,n}(G),
\]
on which we define a height function $h \colon \mathcal{F}_{d,1}(G)\to \mathbb{N}$ that takes an element $[T,(\sigma,f),F]\in \mathcal{F}_{d,1,n}(G)$ to $n$.
Note also that the left $V_d(G)$-action and the poset structure descend to a well-defined action and poset structure on the set
\[
    \mathcal{P}_d(G):= \bigsqcup_{n\geq 1} \mathcal{F}_{d,1,n}(G)/(S_n\wr G),
\]
where $S_n\wr G$ acts on the right on $\mathcal{F}_{d,1,n}(G)$ via its embedding
\[
\sigma(f_1,\ldots,f_n) \mapsto [1_n,\sigma(f_1,\ldots,f_n),1_n]
\]
into $\mathcal{F}_{d,n,n}(G)$.
The left $V_d(G)$-action on $\mathcal{P}_d(G)$ further induces an cellular action on the geometric realization $|\mathcal{P}_d(G)|$ of the poset and we can extend the height function $h$ affine linearly to an invariant height function on $h \colon |\mathcal{P}_d(G)|\to [0,\infty)$.
For $x,y\in \mathcal{P}_d(G)$ we write
\[
[x,y]=\left\{z\in \mathcal{P}_d(G)\mid x\leq z \leq y\right\}
\]
to denote the closed interval of vertices between $x$ and $y$ (and analogously $(x,y)$, $(x,y]$ and $[x,y)$). 
If $x\preccurlyeq y$, then the elements of $[x,y]$ form a Boolean lattice, since the process of performing the simple splits making up an elementary split is commutative. Thus, in this case the subcomplex of $|\mathcal{P}_d(G)|$ spanned by the vertices of $[x,y]$ naturally forms a cube with vertices the elements of $[x,y]$. One can check that any two such cubes intersect in a cube, implying that the subcomplex of $|\mathcal{P}_d(G)|$ consisting of simplices of the form $x_0< x_1 < \ldots < x_n$ and $x_0\preccurlyeq x_n$ admits the structure of a cube complex.
We will denote this cube complex by $\XX_d(G)$ and call it the \emph{Stein--Farley complex} of $V_d(G)$. We obtain an induced height function $h: \XX_d(G)\to [0,\infty)$ which is a Morse function in the sense of Bestvina and Brady \cite{BestvinaBrady97}.

\subsection{Homotopical and homological finiteness properties of $V_d(G)$}
The homotopical finiteness properties $F_n$ of $V_d(G)$ were shown in~\cite[Section 4]{SkipperWitzelZaremsky19} to be at least as good as those of $G$.
Here we will explain how a proof of this fact given in \cite[Section 3]{WuWuZhaoZhou25} for the homotopical finiteness properties $F_n$ in the case $d=2$ in conjunction with a homological version of Brown's lemma readily generalizes to a proof of the following result.

\begin{proposition}
\label{prop:positive-finiteness-properties}
    Let $G\leq \Aut(\TT_d)$ be a self-similar group. If $G$ is of type $\FP_n(R)$ for some unital commutative ring $R$, then $V_d(G)$ is of type $\FP_n(R)$.   
\end{proposition}

The approach pursued in \cite{WuWuZhaoZhou25} is a standard approach which has been applied to many Thompson-like groups and our proof is a straight-forward generalization of theirs. We will thus only sketch the main steps of their proof and refer to their work for further details. Analogues of some of the results in this section were already proven in the context of homotopical finiteness properties of cloning systems, see~\cite[Section 3]{SkipperZaremsky24}.
For the reader not familiar with this approach, we also refer to \cite{Zaremsky21} for a detailed proof in the special case of Thompson's group $F$ that illustrates the main ideas. The proof of finiteness properties in \cite{SkipperWitzelZaremsky19} relies on previous work of Witzel \cite{Witzel19}; it is written in the context of homotopical finiteness properties, but ultimately employs Brown's criterion, which also holds for homological finiteness properties. Thus, a careful check of those two works should also show \Cref{prop:positive-finiteness-properties}. 
We start by recalling Brown's criterion in the homotopical and homological version.

\begin{proposition}\label{prop:Browns-criterion}
    Let $G$ be a group, let $X$ be a contractible affine $G$-CW complex and let $R\neq 0$ be a unital commutative ring. Assume that there is $n \in \mathbb{N}$ such that all cell stabilisers of $k$-cells are $\F_{n-k}$ (resp.\ $\FP_{n-k}(R)$). Let $h \colon X \to \R$ be a $G$-invariant height function. Assume that for every $a\in \R$ the sublevel set $X^{\leq a}=h^{-1}((-\infty,a])$ is $G$-cocompact and that there exists some $r\in \R$ such that the descending link of every vertex $x\in X$ with $h(x)\geq r$ is $(n-1)$-connected (resp.\ homologically $(n-1)$-connected over $R$). Then $G$ is of type $F_n$ (resp.\ of type $\FP_n(R)$). 
\end{proposition}
\begin{proof}
    This is a well-known consequence of combining Brown's criterion \cite{Brown87b} and Bestvina--Brady Morse theory \cite{BestvinaBrady97}, see e.g.~\cite[Proposition 5.15]{GenevoisLonjouUrech22} and \cite[Theorem 5.1]{BuxLlosaIsenrichWu24} for homotopical versions. For the homological versions one checks that homological analogues of all results required for the homotopical version appear in \cite{Brown87b} and  \cite{BestvinaBrady97}.
\end{proof}

To deduce the finiteness properties of $V_d(G)$ and thus prove \Cref{prop:positive-finiteness-properties}, we will apply \Cref{prop:Browns-criterion} to  $h: \XX_d(G)\to \R$. We will sketch the main steps of the argument, following \cite[Section 3]{WuWuZhaoZhou25}.

\begin{proposition}\label{prop:simple-connectivity-and-cocompactness}
    For any self-similar group $G\leq \Aut(\TT_d)$ the complexes $|\mathcal{P}_d(G)|$ and $\XX_d(G)$ satisfy the following properties:
    \begin{enumerate}
        \item $|\mathcal{P}_d(G)|$ and $\XX_d(G)$ are contractible
        \item for every $n\in \mathbb{N}$ the $V_d(G)$-action on the subcomplex $\XX_d(G)^{\leq n}$ spanned by the vertices of height $\leq n$ is cocompact and transitive on vertices of a fixed height.
    \end{enumerate}
\end{proposition}
\begin{proof}[Sketch of proof]
   The proofs are standard. We thus only indicate the main steps and refer to \cite[Section 3.3]{WuWuZhaoZhou25} and the references therein for details. The contractibility of $|\mathcal{P}_d(G)|$ follows from the fact that the poset $(\mathcal{P}_d(G),\leq)$ is directed, that is, for any $x,y\in \mathcal{P}_d(G)$ there is a $z\in \mathcal{P}_d(G)$ with $z\geq x,y$ (see \cite[Lemma 3.6]{WuWuZhaoZhou25}). To see the contractibility of $\XX_d(G)$ one first uses an argument based on a result of Quillen to show that intervals $(x,y)$ with $x\leq y$ and $x\not\preccurlyeq y$ are contractible (see \cite[Proposition 3.8]{WuWuZhaoZhou25}). Since $|\mathcal{P}_d(G)|$ can be obtained from $\XX_d(G)$ by attaching intervals along contractible subspaces, this implies that the inclusion $\XX_d(G)\hookrightarrow |\mathcal{P}_d(G)|$ is a homotopy equivalence (see \cite[Proposition 3.9]{WuWuZhaoZhou25}), proving (1). For (2) one checks explicitly that the $V_d(G)$-action on $\XX_d(G)$ has a single orbit of vertices of height $k\in \mathbb{N}$ for every $k$, which readily implies the assertion (see \cite[Proposition 3.7]{WuWuZhaoZhou25}). 
\end{proof}

The only place where the finiteness properties of $G$ enter the proof of the finiteness properties of $V_d(G)$ is via the finiteness properties of cell stabilisers. We thus give a detailed proof of the computation of the cell stabilisers.

\begin{proposition}\label{prop:cell-stabilisers}
    Let $G\leq \Aut(\TT_d)$ be self-similar of type $F_m$ (resp.\ $\FP_m(R)$). Then the stabiliser of a cube $[x,y]$ in $\XX_d(G)$ with $h(x)=n$ is a finite index subgroup of $S_n\wr G$ and, in particular, of type $F_m$ (resp.\ $\FP_m(R)$).
\end{proposition}
\begin{proof}
    The proof is analogous to the proofs of \cite[Lemma 3.10 and 3.11]{WuWuZhaoZhou25}. We first prove the assertion for vertices.
    Let $x\in \XX_d(G)$ be a vertex with $h(x)=n$. Then $x$ is represented by an element of the form $[T,(\sigma,f),F]\in \mathcal{F}_{d,1,n}(G)$. Let $\gamma=[T_-,(\nu,g),T_+]\in V_d(G)$ with $\gamma x = x$. Equivalently, there is an element $[1_n,(\mu,h),1_n] \in S_n \wr G$ such that the following identity holds in $\mathcal{F}$:
    \begin{align*}
        [T,(\sigma,f),F]&= [T_-,(\nu,g),T_+][T,(\sigma,f),F][1_n,(\mu,h),1_n],
    \end{align*}
    which in turn means that $\gamma$ has the form
    \begin{equation}\label{eq:point-stabiliser}
        \gamma=[T_-,(\nu,g),T_+] = [T,(\sigma,f),F][1_n,(\mu,h)^{-1}, 1_n][F,(\sigma,f)^{-1},T].
    \end{equation}
    We deduce that ${\rm Stab}_{V_d(G)}(x) \cong S_n\wr G$, which is of type $\F_m$ (resp.\ $\FP_m(R)$) if and only if $G$ is.

    Now assume that $[x,y]$ is a cube in $\XX_d(G)$ with $h(x)=n$. Then $${\rm Stab}_{V_d(G)}([x,y])={\rm Stab}_{V_d(G)}(x)\cap {\rm Stab}_{V_d(G)}(y)\leq {\rm Stab}_{V_d(G)}(x) \cong S_n\wr G$$ and we will show that this inclusion is of finite index. Since $x\preccurlyeq y$, there is a representative $[T,(\sigma,f),F]$ for $x$ and an elementary forest $F_J=F_J^{(n)}$, such that
    \[
    [T,(\sigma,f),F][F_J,(\id,\id),1_{n+(d-1)k}]
    \]
    is a representative for $y$. We argue as above that any $\gamma\in {\rm Stab}_{V_d(G)}(x)$ is as in \Cref{eq:point-stabiliser}. If $\gamma y = y$ there has to be some $(\alpha,r)\in S_{n+(d-1)k}\wr G$ such that (after left multiplying by $[F,(\sigma,f)^{-1},T]$) we have 
    \begin{align*}
        &\ [1_n,(\mu,h)^{-1},1_n][F_J,(\id,\id),1_{n+(d-1)k}]\\=&\ [F_J,(\id,\id),1_{n+(d-1)k}]\cdot [1_{n+(d-1)k},(\alpha,r),1_{n+(d-1)k}]\\
        =&\ [F_J,(\alpha,r),1_{n+(d-1)k}].
    \end{align*}
    This is the case if and only if $\mu(J)=J$. Thus,
    \[
        {\rm Stab}_{V_d(G)}([x,y])\cong \left\{(\mu,h)\in S_n\wr G\mid \mu(J)=J\right\},
    \]
    which is a finite index subgroup of $S_n\wr G$.
\end{proof}

The following result is the last remaining ingredient that we need in order to prove Proposition~\ref{prop:positive-finiteness-properties}.

\begin{proposition}\label{prop:connectivity-of-descending-links}
Let $G\leq \Aut(\TT_d)$ be self-similar. Then for every $k\in \N$ there is a number $n_k\in \N$ such that for every $x\in \XX_d(G)$ with $h(x)\geq n_k$ the descending link $\lkd (x)$ is $k$-connected.
\end{proposition}
\begin{proof}[Sketch of proof]

In analogy to \cite[Section 3.4]{WuWuZhaoZhou25} this follows by reduction to the connectivity properties of a certain matching complex $\mathcal{M}_{d,n}$. The reduction step is completely analogous to the one in \cite{WuWuZhaoZhou25}, modulo checking that one can replace $2$-element sets by $d$-element sets everywhere (in particular it is independent of the finiteness properties of $G$). We will thus only indicate the main steps of the reduction argument and refer the reader to \cite[Section 3.4]{WuWuZhaoZhou25} for details.
One first proves that for $x\in \XX_d(G)$ with $h(x)=n$ the descending link is isomorphic to the flag complex $\mathcal{E}_{d,n}(G)$ whose $(|J|-1)$-simplices are elements of the form
\[
[1_n,(\sigma,f),F_J^{(n-(d-1)|J|)}]\in \mathcal{F}_{d,n,n-(d-1)|J|}(G)/(S_{n-(d-1)|J|}\wr G)
\]
with face inclusion
\[
[1_n,(\sigma,f),F_J^{(n-(d-1)|J|)}]\subseteq [1_n,(\sigma',f'),F_{J'}^{(n-(d-1)|J'|)}]
\]
if and only if
\[
[1_n,(\sigma',f'),F_{J'}^{(n-(d-1)|J'|)}]\leq [1_n,(\sigma,f),F_J^{(n-(d-1)|J|)}].
\]

Then one verifies that there is a well-defined map to the matching complex $\mathcal{M}_{d,n}$ of $d$-element subsets of $\left\{1,\ldots,n\right\}$, which takes the $(|J|-1)$-simplex $$[1_n,(\sigma,f),F_J^{(n-(d-1)|J|)}]$$ to the $(|J|-1)$-simplex \footnote{Recall that the matching complex $\mathcal{M}_{d,n}$ is the simplicial complex whose vertices are the $d$-element subsets of $\left\{1,\ldots,n\right\}$ and whose $(k-1)$-simplices are $k$-tuples of pairwise disjoint $d$-element subsets of $\left\{1,\ldots,n\right\}$.} \footnote{Note that there is a slight difference here with respect to the map in \cite{WuWuZhaoZhou25}. This is because we are following the convention in \cite{SkipperWitzelZaremsky19}, where the permutation $\sigma$ maps leaves of $F_J^{(n-(d-1)|J|)}$ to $1_n$ while \cite{WuWuZhaoZhou25} map leaves of $1_n$ to $F_J^{(n-(d-1)|J|)}$. }
\[
	\left\{\left\{\sigma(i_1),\ldots,\sigma(i_d)\right\}\mid \mbox{ there is a caret in $F_J^{(n-(d-1)|J|)}$ with leaves $\left\{i_1,\ldots, i_d\right\}$}\right\}.
\]
For the details on the above steps of the proof, we refer to the discussion in the beginning of \cite[Section 3.4]{WuWuZhaoZhou25}.

Next one shows (see \cite[Lemmas 3.13 and 3.14]{WuWuZhaoZhou25}) that the map $\pi: \mathcal{E}_{d,n}(G)\to \mathcal{M}_{d,n}$ is a complete join, that is, $\pi$ is surjective and for every simplex $\alpha=\left\{v_0,\ldots, v_k\right\}$ of $\mathcal{M}_{d,n}$ the preimage is $\pi^{-1}(\alpha)=\pi^{-1}(v_0)\ast\ldots\ast\pi^{-1}(v_k)$. 
One can now deduce connectivity properties of $\mathcal{E}_{d,n}(G)$ from those of $\mathcal{M}_{d,n}$ by applying \cite[Proposition 3.5]{HatcherWahl10}. For this one needs to show that $\mathcal{M}_{d,n}$ is weakly Cohen--Macaulay of dimension going to $\infty$ as $n\to \infty$ (see \cite[Proof of Corollary 3.15]{WuWuZhaoZhou25} for the case $d=2$).
This is well-known (see e.g. \cite[Theorem 4.1 and Corollary 4.2]{BjornerLovaszetal94}).
For the readers convenience we include a brief proof based on the work of Belk and Forrest \cite{BelkForrest19}. By definition, for every simplex $\alpha$ of $\mathcal{M}_{d,n}$, every vertex $v\in \mathcal{M}_{d,n}$ is connected to all but at most $d$ vertices of $\alpha$. Moreover, the maximal dimension of a simplex in $\mathcal{M}_{d,n}$ is $\left\lfloor \frac{n}{d}\right\rfloor -1$.
It follows that the complex $\mathcal{M}_{d,n}$ is $(\left\lfloor \frac{n}{d}\right\rfloor -1,d)$-grounded in the sense of \cite[Definition 4.8]{BelkForrest19} and thus $(\left\lfloor \frac{n-d}{d^2}\right\rfloor -1)$-connected by \cite[Theorem 4.9]{BelkForrest19}.
Moreover, the link of a $k$-simplex in $\mathcal{M}_{d,n}$ is isomorphic to $\mathcal{M}_{d,n-dk}$, thus $(\left\lfloor \frac{n-dk-d}{d^2}\right\rfloor -1)$-connected. 
In particular, $\mathcal{M}_{d,n}$ is weakly Cohen--Macaulay of dimension $\left\lfloor \frac{n-d}{d^2}\right\rfloor -1$. Thus, \cite[Proposition 3.5]{HatcherWahl10} implies that for every $x\in \N$ with $h(x)=n$ the descending link $\lkd(x)\cong \mathcal{E}_{d,n}(G)$ is $(\left\lfloor \frac{n-d}{d^2}\right\rfloor -1)$-connected. In particular, the connectivity goes to $\infty$ as $n\to \infty$, completing the proof. 

\end{proof}

\begin{proof}[Proof of \Cref{prop:positive-finiteness-properties}]
    By \Cref{prop:simple-connectivity-and-cocompactness,prop:cell-stabilisers,prop:connectivity-of-descending-links}, this is now a direct consequence of applying \Cref{prop:Browns-criterion} to $h: \XX_d(G)\to [0,\R)$. 
\end{proof}

\section{Quasi-retracts and Dehn functions}\label{sec:quasi-retracts}

Recall first that a group $Q$ is called a \emph{retract} of a group $G$ if there is a pair of group homomorphisms
\[ Q \overset{i}{\hookrightarrow} G \overset{r}{\twoheadrightarrow} Q\]
such that $r\circ i$ is the identity on $Q$. Suppose $Q$ is a retract of $G$. Then if $G$ is of type $\F_n$, so is $Q$, see for example \cite[Proposition 4.1]{Bux04}.  The same holds if one replaces retract by quasi-retract. Let us make this precise.
Recall that a function $f\colon X\to Y$ between two metric spaces is said to be \emph{coarse Lipschitz}  if there exist constants $C,D>0$ so that 
\[
d(f(x),f(x')) \leq Cd(x,x')+D \text{ for all } x,x'\in X.
\]
For example, any homomorphism between finitely generated groups is coarse Lipschitz with respect to the word metrics.
A function $\rho \colon X\to Y$ is said to be a quasi-retraction if it is coarse Lipschitz and there exists a coarse Lipschitz function $\iota \colon Y\to X$ and a constant $E>0$ so that $d(\rho\circ \iota(y),y)\leq E$ for all $y\in Y$. If such a function exists, $Y$ is said to be a quasi-retract of $X$. 

\begin{theorem}\cite[Theorem 8]{Alonso94}\label{thm-quasi-re-fin}
Let $G$ and $Q$ be finitely generated groups such that $Q$ is a quasi-retract of $G$ with respect to word metrics corresponding to some finite generating sets of $G$ and $Q$.
Then
\begin{enumerate}
    \item if $G$ is of type $\F_n$, so is $Q$,
    \item if $G$ is of type $\FP_n(R)$ for some unital commutative ring $R\neq 0$, so is $Q$.
\end{enumerate}
\end{theorem}
\begin{proof}
   In \cite[Theorem 8]{Alonso94}, assertion (2) is only stated for $R=\mathbb{Z}$. However, the only place where $R=\mathbb{Z}$ is used in the proof is when \cite[Theorem 2.2]{Brown87b} is invoked. Since \cite[Theorem 2.2]{Brown87b} holds for arbitrary rings $R$ (see Brown's Remark (1) right after its statement in \cite{Brown87b}), one checks readily that so does \cite[Theorem 8]{Alonso94} with the same proof. This has also been observed in \cite[Theorem 5.5]{CastellanoCorobCook20} in the more general context of compactly generated tdlc groups.
\end{proof}

Another interesting consequence of quasi-retracts is that they provide a lower bound for Dehn functions. Given a finitely presented group $G$ with finite presentation $\langle  X \mid R \rangle$, let $w$ be a word in the free group $F(X)$ generated by elements in $X$. If $w=1$ in $G$, we can write $w$ as finite product of conjugates of elements of relations from $R$ and their inverses, i.e.
\[ w = \prod_{i=1}^n w_ir_iw_i^{-1}\]
for some $r_i\in R^{\pm 1}$, $w_i\in F(X)$. Define $\Area(w)$ to be the minimal such $n$. The function \
\[ \delta_G(n) = \max \{ \Area(w) \mid |w|\leq n, w=1 ~\text{in}~ G\}\]
is called the \emph{Dehn function} of $G$.

The Dehn functions associated to different finite presentations of the same group are $\cong$ equivalent in the following sense: given functions $f_1,f_2:\mathbb{N} \to [0,\infty)$, one writes $f_1\preceq f_2$ if there is a constant $K>0$ such that $f_1(n) \leq Kf_2(Kn)+Kn+K$, and one writes $f_1\cong f_2$ if $f_1\preceq f_2$ and $f_2 \preceq f_1$.

For a null-homotopic word $w$, one can interpret $\Area(w)$ in terms of the (cellular) filling area for the loop  described by $w$, based at the identity, in the presentation complex corresponding to $\langle X \mid R\rangle$. If $G$ is of type $\F_{r+1}$, one can generalise this geometric viewpoint and define higher Dehn functions $\delta^{(k)}_G(n)$ for $1\leq k \leq r$ in terms of filling volumes of cellular $k$-spheres in universal covers of classifying spaces with finite $(r+1)$-skeleton. We refer to \cite{AlonsoWangPride99} for a precise definition, since we will not require it here. Instead we will only use known facts about Dehn functions, including the following behaviour under quasi-retracts.

\begin{theorem}[{\cite{Alonso90},\cite[Corollary 5]{AlonsoWangPride99}}]\label{thm:quasi-retract-dehnfunc}
Let $G$ and $Q$ be groups of type $\F_{n+1}$ such that $Q$ is a quasi-retract of $G$ with respect to word metrics corresponding to some finite generating sets. Then $\delta^{(k)}_Q \preceq \delta^{(k)}_G$ for $1\leq k \leq n$.
\end{theorem}

As noted in the introduction, a self-similar group $G$ does not need to be a quasi-retract of $V_d(G)$.
In~\cite{SkipperWitzelZaremsky19,Zaremsky25}, the authors deal with this fact by introducing a set of conditions that guarantees that $G$ is a quasi-retract of $V_d(G)$ and then find groups that satisfy these conditions.
In this section we follow a different approach.
Instead of realizing a given group $G$ as a quasi-retract of $V_d(G)$, we produce a self-similar split extension $\Gamma$ of $G$ for which the natural retraction $\Gamma \rightarrow G$ can be precomposed with some map $V_d(\Gamma) \rightarrow \Gamma$ such that the resulting map $V_d(\Gamma) \rightarrow G$ is a quasi-retraction.
We start by introducing the following notion.

\begin{definition}\label{def:persistent-retracts}
Let $G \leq \Aut(\TT_d)$ be a self-similar group and let $r \colon G \rightarrow Q$ be a retraction.
We say that $r$ is \emph{persistent} if for every $g$ and every $u\in V(\TT_d)$ we have $r(g) = r(g^u)$, i.e.\ the image of a state of $g$ does not depend on the vertex at which the state is defined.
\end{definition}

We proceed by showing how persistent retractions can be used to construct quasi-retractions.
To this end, recall that we can view $V_d$ naturally as the subgroup of $V_d(G)$ with trivial labeling, i.e.\ $V_d = V_d(\{1\})$.
Recall also that $C$ denotes the unique tree with one root and $d$ leaves, and we have an embedding $\iota_1 \colon G \to V_d(G)$ by:
\[
h \mapsto (C,\mathrm{id}(h,1,\cdots, 1),C).
\]

\begin{lemma}\label{lem:pers-retr-implies-quasi-retr}
Let $G \leq \Aut(\TT_d)$ be a finitely generated self-similar group and let $r \colon G \rightarrow Q$ be a persistent retraction.
Then the R\"over--Nekrashevych group $V_d(G)$ admits $Q$ as a quasi-retract.
\end{lemma}
\begin{proof}

We first define a map $\rho: V_d(G) \to Q$ and then prove that it is a quasi-retract. Given any element $x = [T_-, \sigma(g_1,\cdots,g_n), T_+]$, let $\rho(x) := r(g_1)$. To prove that $\rho$ is well-defined, it suffices to show that $\rho(x)$ is invariant under elementary expansion of $[T_-, \sigma(g_1,\cdots,g_n), T_+]$. If the elementary expansion is performed at the first leaf, this is the case since $r$ is persistent. If the elementary expansion is not performed at the first leaf, $g_1$ does not change, so $\rho$ is also well-defined in this case.

We now proceed to show that $\rho$ is a quasi-retract. Since $Q$ is a retraction of the finitely generated group $G$, we can write $G = \ker(r) \rtimes Q$ and $Q$ is finitely generated by some finite generating set $S_Q$. We can choose a finite generating set $S_G$ of $G$ such that $S_G\setminus S_Q \subseteq \ker(r) $. Let $S_{V_d}$ be a finite generating set of $V_d$.  By Proposition \ref{prop:prpty-Vd(G)} (1), $V_d(G)$ is generated by $S_{V_d}$ and $\iota_1(S_G)$. It suffices to show that the map $\rho: V_d(G) \to Q$ is coarse Lipschitz with respect to the corresponding word metrics on $V_d(G)$ and $Q$, where as in \cite{SkipperWitzelZaremsky19} we work with left Cayley graphs. To see this, observe that persistence of $r$ and a straight-forward computation analogous to the one in the proof of \cite[Proposition 5.5]{SkipperWitzelZaremsky19} yield:
 \begin{enumerate}
    \item  $\rho( \iota_1((h,s)) x)\in \{\rho(x),  s\rho(x)\}$ for all $s\in S_Q, h\in \ker(r)$ and $x\in V_d(G)$,  and
    \item  $\rho(yx) = \rho(x)$ for any $y\in V_d$ and $x\in V_d(G)$.
\end{enumerate}
It follows that $\rho$ is non-expanding and hence coarse Lipschitz. 
\end{proof}

\section{Creating finite abelianization}\label{sec:finite-abel}
In this section, we will apply a construction by Zaremsky \cite{Zaremsky25embedding} to a finitely generated self-similar group $G\leq \Aut(\TT_d)$ to produce an embedding $j: G\to \Aut(\TT_{md})$ so that $V_{md}(j(G))$ has finite abelianization. We will then show that this construction preserves persistent retractions. 

We start by explaining Zaremsky's construction. For $g\in G$, let $g=\rho(g)(g_1,\ldots,g_d)$ be its wreath recursion. Given $m\in \N$, define a morphism 
\[
f^{(m)}: G \to \Aut(\TT_{md}),~ g\mapsto f_g^{(m)} 
\]
via the wreath recursion
\[
f_g^{(m)}=\rho(g)^{\oplus m}(f_{g_1}^{(m)},\ldots, f_{g_{d}}^{(m)},f_{g_1}^{(m)}\ldots,f_{g_{d}}^{(m)},\ldots, f_{g_1}^{(m)},\ldots,f_{g_{d}}^{(m)}),
\]
where $\rho(g)^{\oplus m}\in S_{md}$ is given by
\[
\rho(g)^{\oplus m}(kd+i)=kd+\rho(g)(i) \text{ for } 0\leq k \leq d-1 \text{ and } 1\leq i \leq d.
\]
The morphism $f^{(m)}$ is clearly well-defined and has self-similar image. It is injective, since the induced $G$-action restricts to the original $G$-action on the subtree $\TT_d\subset \TT_{md}$ defined by the inclusion $\left\{1,\ldots,d\right\}\hookrightarrow \left\{1,\ldots, md\right\}$.

\begin{lemma}\label{lem:finite-abelianization}
    Let $G\leq \Aut(\TT_d)$ be a finitely generated self-similar group. Then there is some $m\in \N$ such that $V_{md}(f^{(m)}(G))_{ab}$ is finite.
\end{lemma}
\begin{proof}
    This follows from Zaremsky's proof of \cite[Theorem 1.1]{Zaremsky25embedding}. We include his argument here for the readers convenience. Denote by $\pi: G \to G_{ab}$ the abelianization map. By \Cref{prop:abelianization-Vd(G)}, for $m$ even, we have
    \[
        V_{md}(f^{(m)}(G))_{ab} \cong G_{ab}/\left\langle \pi(g) - m \cdot (\pi(g_1)+\ldots + \pi(g_d))\mid g\in G\right\rangle.
    \]
    The right side is a quotient of the abelian group
    \begin{align*}
        & = \oplus_{i=1}^r \Z \pi(a_i)/\left\langle \pi(a_i) - m \cdot (\pi(a_{i,1}) + \ldots + \pi(a_{i,d}))\mid 1\leq i \leq r \right\rangle\\
        &= \Z^r/ {\rm Im}(I_r-m\cdot A),
    \end{align*}
    where $A$ is the $(r\times r)$-matrix whose $i$-th column is $\pi(a_{i,1}) + \ldots + \pi(a_{i,d})$ expressed with respect to the basis $\left\{\pi(a_1),\ldots,\pi(a_r)\right\}$ of $\oplus_{i=1}^r \Z \pi(a_i)$. Since $A$ only has finitely many eigenvalues, there is some $m\in \N$ such that ${\rm Im}(I_r-mA)\leq \Z^r$ has finite index and thus so that $V_{md}(f^{(m)}(G))_{ab}$ is finite.
\end{proof}

\begin{lemma}\label{lem:pers-retr-inv-under-doubling}
Let $G \leq \Aut(\TT_d)$ be a self-similar group and let $r \colon G \rightarrow Q$ be a persistent retraction.
Then $r^{(m)} \colon f^{(m)}(G) \rightarrow Q,\ f^{(m)}_g \mapsto r(g)$ is a persistent retraction of $f^{(m)}(G) \leq \Aut(\TT_{md})$.
\end{lemma}
\begin{proof}
The map $r^{(m)}:f^{(m)}(G)\to Q$ is a well-defined retraction, since it is the composition of the isomorphism $ f^{(m)}(G)\to G$, $f_g^{(m)}\mapsto g$, with the retraction $r:G\to Q$. We claim that $r^{(m)}$ is also persistent. Indeed, for $g\in G$, the wreath recursions $g=\rho(g)(g_1,\ldots, g_d)$,
\[
    f_g^{(m)}=\rho(g)^{\oplus m}(f_{g_1}^{(m)},\ldots, f_{g_{d}}^{(m)},f_{g_1}^{(m)}\ldots,f_{g_{d}}^{(m)},\ldots, f_{g_1}^{(m)},\ldots,f_{g_{d}}^{(m)}),
\]
and persistence of $r$, imply that for every $1\leq j\leq md$ there is some $1\leq i \leq d$ such that
\[
r^{(m)}( (f_g^{(m)})_{j})
=r^{(m)}(f^{(m)}_{g_i})
=r(g_i)
=r(g)
=r^{(m)}(f_g^{(m)}).
\]
\end{proof}

\begin{corollary}\label{cor:from-pers-retr-to-virt-simpl-and-qr}
Let $G \leq \Aut(\TT_d)$ be a finitely generated self-similar group and let $r \colon G \rightarrow Q$ be a persistent retraction.
Then there is some $m \in \N$ and some embedding $j \colon G \rightarrow \Aut(\TT_{md})$ such that $j(G)\leq \Aut(\TT_{md})$ is a self-similar subgroup such that the corresponding R\"over--Nekrashevych group $V_{md}(G)$ has finite abelianization and admits $Q$ as a quasi-retract.
\end{corollary}
\begin{proof}
    This is a direct consequence of \Cref{lem:finite-abelianization,lem:pers-retr-inv-under-doubling,lem:pers-retr-implies-quasi-retr}.
\end{proof}

\section{Realizing linear groups as persistent retracts}\label{sec:realizing-persistent-retracts}

The goal of this section is to provide for a given finitely generated subgroup $Q \leq \GL_n(\Q)$, a self-similar group $G$ that admits a persistent retraction $r \colon G \rightarrow Q$ such that the finiteness properties of $G$ and $Q$ coincide.
The self-similarity of $G$ will be inherited from the affine group $\mathrm{Aff}_n(\Z[1/N]) \defeq \Z[1/N]^n \rtimes \mathrm{GL}_n(\Z[1/N])$, which is known to admit a natural family of self-similar actions, see e.g.~\cite{KionkeSchesler23amenable,SunicVentura2012}.
Let us start by recalling the latter.

\subsection{Self-similar actions of $\mathrm{Aff}_n(\Z[1/N])$}

Let $p$ be a prime that does not divide $N$.
We consider the set $X_{p,n} \defeq \{0,\ldots,p-1\}^n$, which will play the role of our alphabet $X_d$ from the previous sections.
Let $\TT_{p,n}$ denote the rooted tree corresponding to $X_{p,n}$.
Recall that the boundary $\partial \TT_{p,n}$ can be identified with the set $X_{p,n}^{\omega}$ of infinite sequences over $X_{p,n}$, which in turn can be identified with $\Z_p^n$ via $(x_n) \mapsto \sum \limits_{n=0}^{\infty} p^i x_i$.
In view of this, the natural action of
\[
\Aff_n(\Z_p) = \Z_p^n \rtimes \mathrm{GL}_n(\Z_p)
\]
on $\Z_p^n$ induces an action on $\TT_{p,n}$ and hence on $X_{p,n}^{\omega}$.
Since this action is moreover faithful, we will from now on identify $\Aff_n(\Z_p)$ with its image in $\Aut(\TT_{p,n})$.
With this point of view, it was shown in~\cite[Lemma 6.3]{KionkeSchesler23amenable} that the subgroup
\[
\Aff_n(\Z) = \Z^n \rtimes \mathrm{GL}_n(\Z)
\]
of $\Aff_n(\Z_p)$ is self-similar.
More generally, essentially the same proof as in~\cite{KionkeSchesler23amenable} can be used to deduce the self-similarity of $\Aff_n(R)$ for more general rings $R$, such as $R = \Z[1/N]$.
In what follows it will be convenient to make this more precise.

\begin{definition}\label{def:p-rigid}
We call a subring $R$ of $\Z_p$ \emph{rigid} if for every $x \in \Z_p$ we have $x \in R$ if and only if $px \in R$.
In other words, $R$ is a rigid subring of $\Z_p$ if $(p)R = (p)\Z_p \cap R$.
\end{definition}

\begin{lemma}\label{lem:rigid-self-similar}
Let $R$ be a rigid subring of $\Z_p$ and let $g \in \Aff_n(R)$ be an element that is given by $g(v) = Av+b$ for some $A \in \GL_n(R)$, $b \in R^n$ and every $v \in \TT_{p,n}$.
Then for every $x \in X_{p,n}$ there is some $b' \in R^n$ such that the state $g^x$ is given by $g^x(v) = Av+b'$ for every $v \in \TT_{p,n}$.
\end{lemma}
\begin{proof}
Let $u \in \Z_p^n$ be an element of the form $u = x + pw$ with $x \in X_{p,n}$ and $w \in \Z_p^n$.
Let further $x' \in X_{p,n}$ and $b' \in \Z_p^n$ be such that $Ax+b = x' + pb'$.
Then we have
\[
g(u)
= A(x + pw) + b
= Ax+b + pAw
= x' + p(Aw+b'),
\]
which tells us that $g^x$ is given by $g^x(w) = Aw+b'$.
Since $\Z\subset R$, we have $X_{p,n}\subset R^n$, where we identify $x\in X_{p,n}$ with the sequence $(x,0,0,\ldots)\in \Z_p^n$.
Using $Ax+b = x' + pb'$ and our assumption that $R$ is rigid we thus deduce that
\[
pb' = Ax+b - x' \in (p) \Z_p^n \cap R^n = (p)R^n. 
\] 
Hence $b' \in R^n$, which completes the proof. 
\end{proof}

\begin{lemma}\label{lem:rigid-self-similar-and-retract}
Let $R$ be a rigid subring of $\Z_p$ and let $Q \leq \GL_n(R)$ be a subgroup.
Then $R^n \rtimes Q \leq \Aff_n(R)$ is a self-similar subgroup of $\Aut(\TT_{p,n})$ for which the canonical map $R^n \rtimes Q \rightarrow Q$ is a persistent retraction.
\end{lemma}
\begin{proof}
Let $g \in R^n \rtimes Q$ be an element that is given by $g(v) = Av+b$ for some $A \in Q$ and $b \in R^n$.
Since $(g^x)^y = g^{xy}$ for all $x,y \in X_{p,n}$, it follows by induction from Lemma~\ref{lem:rigid-self-similar} that for every $u \in X_{p,n}^{\ast}$ there is some $b' \in R^n$ such that the state $g^u$ is given by
\begin{equation}\label{eq:rigid-self-similar}
g^u(v) = Av+b'
\end{equation}
In particular it follows that $g^u \in R^n \rtimes Q$, which implies that $R^n \rtimes Q$ is a self-similar subgroup of $\Aut(\TT_{p,n})$.
To prove the second claim of the corollary, we consider the canonical retraction $r \colon R^n \rtimes Q \rightarrow Q$.
From~\eqref{eq:rigid-self-similar}, it directly follows that $r(g^u) = A$ for every $u \in X_{p,n}^{\ast}$, which completes the proof.
\end{proof}

Recall that we have chosen $p$ to be a prime that does not divide $N$.
Thus $N$ is a unit in $\Z_p$ so that we can think of $\Z[1/N]$ as a subring of $\Z_p$, which is easily seen to be rigid.
Using Lemma~\ref{lem:rigid-self-similar-and-retract}, we can now easily deduce the following.
For convenience, we recall the notation that appeared so far.

\begin{corollary}\label{cor:aff-Z-N-self-sim}
For every finitely generated subgroup $Q$ of $\GL_n(\Q)$, there exists a number $N \in \N$ and a prime $p$ such that $\Z[1/N]^n \rtimes Q \leq \Aff_n(\Z[1/N])$ is a self-similar subgroup of $\Aut(\TT_{p,n})$ for which the canonical map $\Z[1/N]^n \rtimes Q \rightarrow Q$ is a persistent retraction.
\end{corollary}
\begin{proof}
    Let $S_Q$ be some finite generating set for $Q$. Then there is some $N\in \N$ such that $S_Q\subset \GL_n(\Z[1/N])$ and thus $Q\leq \GL_n(\Z[1/N])$. For this choice of $N$, the assertion is a direct consequence of Lemma~\ref{lem:rigid-self-similar-and-retract}.
\end{proof}

\subsection{Extending linear groups with Baumslag--Solitar groups}\label{subsec:self-similar-extensions}

Our next goal is to show that every finitely generated subgroup $G$ of $\GL_n(\Q)$ arises as a persistent retract of a self-similar group that has the same finiteness properties as $G$.
To this end, we collect some auxiliary observations.
To formulate the first of them, we write $D_n(\Q)=\left\{\diag(\lambda)\mid \lambda\in \Q^{\times}\right\}$ for the subgroup of $\GL_n(\Q)$ with elements the diagonal matrices $\diag(\lambda)$ with all entries equal to $\lambda\in \mathbb{Q}^{\times}$.

\begin{lemma}\label{lem:disjoint-embedding}
For every finitely generated subgroup $G \leq \GL_n(\Q)$ there are natural numbers $m,N \in \N$ and a monomorphism $\iota \colon G \rightarrow \GL_m(\Z[1/N])$ such that $\iota(G) \cap D_m(\Q) = \{\id\}$.
\end{lemma}
\begin{proof}
As explained in the last subsection, there exists a natural number $N \in \N$ such that $G \leq \GL_n(\Z[1/N])$.
By applying the morphism
\[
j \colon \GL_n(\Z[1/N]) \rightarrow \SL_{n+1}(\Z[1/N]),\ A \mapsto \begin{pmatrix}
A & 0\\
0 & \det(A)^{-1}
\end{pmatrix}
\]
to $G$ we obtain the subgroup $j(G) \leq \SL_{n+1}(\Z[1/N])$ that is isomorphic to $G$.
Suppose $n+1$ is odd and let $\diag(\lambda) \in \SL_{n+1}(\Z[1/N])$ be the diagonal matrix all of whose diagonal entries are equal to $\lambda$.
Then $\lambda^{n+1} = 1$ and hence $\lambda \in \{-1,1\}$.
If $n+1$ is odd, then it follows that $\lambda = 1$ and hence that $D_m(\Q) \cap \iota(G) = \{\id\}$.
If $n+1$ is even, then we can apply $j$ a second time and continue as before, which completes the proof.
\end{proof}

The utility of Lemma~\ref{lem:disjoint-embedding} for us is that it allows us to restricts ourselves to the case where the given subgroup $G \leq \GL_n(\Q)$ has trivial intersection with the subgroup $K_{n,N}$ of $\mathrm{Aff}_n(\Z[1/N]) = \Z[1/N]^n \rtimes \mathrm{GL}_n(\Z[1/N])$ that is generated by $\Z[1/N]^n$ and the diagonal matrix $\diag(N)$ whose diagonal entries are equal to $N$.

\begin{lemma}\label{lem:finiteness-properties}
The group $K_{n,N}$ is of type $\F_{\infty}$.
\end{lemma}
\begin{proof}
Consider the embedding
\[
\alpha \colon \Z^n \rightarrow \Z^n,\ v \mapsto Nv
\]
and its associated ascending $\HNN$-extension $\Z^n \ast_{\alpha}$.
Since $\Z^n$ is of type $\F_{\infty}$, it follows from~\cite[II.7.2 Exercise 3]{Geoghegan08} that $\Z^n \ast_{\alpha}$ is of type $\F_{\infty}$.
Thus the lemma follows from the isomorphism\footnote{For $n=1$ this is the well-known isomorphism $BS(1,N)\cong \Z[1/N]\rtimes \Z$. The case $n > 1$ follows readily from the case $n=1$ by restricting the induced isomorphism $BS(1,N)^n\cong (\Z[1/N]\rtimes \Z)^n$.}
\[
\Z^n \ast_{\alpha}
= \langle \ x_1,\ldots,x_n,t \ | \ [x_i,x_j],\ t x_i t^{-1} = x_i^N \ \forall i,j \ \rangle \rightarrow K_{n,N}
\]
that is induced by $x_i \mapsto e_i$ and $t \mapsto \diag(N)$, where $e_i$ denotes the $i$-th standard unit vector in $\Z^n$.
\end{proof}

\begin{proposition}\label{prop:realizing-G-as-persistent-retract}
Every finitely generated subgroup $G$ of $\GL_n(\Q)$ arises as a persistent retract of a self-similar group that has the same finiteness properties as $G$.
\end{proposition}
\begin{proof}
Using Lemma~\ref{lem:disjoint-embedding}, we can assume that $G$ lies in $\GL_n(\Z[1/N])$ and satisfies $G \cap D_n(\Q) = \{\id\}$.
As a consequence, the intersection of $G$ and $K_{n,N}$ in $\Aff_n(\Z[1/N])$ is trivial is well.
Since moreover $K_{n,N}$ is a normal subgroup of $\Gamma \defeq \langle K_{n,N},G \rangle$, we obtain a short exact sequence
\[
1 \rightarrow K_{n,N} \rightarrow \Gamma = K_{n,N} \rtimes G \rightarrow G \rightarrow 1.
\]
In particular, $G$ is a retract of $\Gamma$. In view of \Cref{thm-quasi-re-fin}, Lemma~\ref{lem:finiteness-properties} and the fact that the finiteness properties of an extension of groups $A,B$ are bounded from below by the minimum of the finiteness properties of $A$ and $B$, see e.g.~\cite[II.7.2 Exercise 1]{Geoghegan08}, it follows that $\Gamma$ has the same finiteness properties as $G$.
A further consequence of $G \cap D_n(\Q) = \{\id\}$ is that the group $H \defeq \langle G, \diag(N) \rangle$ splits as a direct product
\[
H \cong G \oplus \langle \diag(N) \rangle \cong G \times \Z
\]
since $\diag(N)$ lies in the center of $\GL_n(\Q)$.
By combining that $H\cap (\Z[1/N])^n=\{1\}$ with the fact that the subgroup $\Z[1/N]^n$ of $K_{n,N}$ is normal in $\Aff_n(\Z[1/N])$, we obtain the short exact sequence
\[
1 \rightarrow \Z[1/N]^n \rightarrow \Gamma \rightarrow H \rightarrow 1.
\]
Thus we see that $\Gamma$ is of the form $\Gamma = \Z[1/N]^n \rtimes H$ so that we can apply Lemma~\ref{lem:rigid-self-similar-and-retract} to deduce that it is a self-similar subgroup of $\Aut(\TT_{p,n})$, $p \nmid N$, for which the canonical map $\Gamma = \Z[1/N]^n \rtimes H \rightarrow H$ is a persistent retraction.
Finally, we can compose the latter retraction with the projection $H \cong G \oplus \langle \diag(N) \rangle \rightarrow G$ to obtain a persistent retraction $\Gamma \rightarrow G$.
\end{proof}

\begin{remark}\label{rem:loosing-Fn}
In general, the transition from a subgroup $G$ of $\GL_n(\Z[1/N])$ to the group $\Z[1/N]^n \rtimes G$ can lead to a loss of finiteness properties.
Indeed, if $G$ is the trivial group, then $\Z[1/N]^n \rtimes G$ is not even finitely generated.
\end{remark}

\section{Applications}\label{sec:applications}

We are now ready to prove the results stated in the introduction.
For the convenience of the reader we repeat them here.

\begin{theorem}\label{thm:linear-groups-as-retracts}
Let $H$ be a finitely generated subgroup of $\GL_n(\Q)$ for some $n \in \N$.
There exists a simple group $G$ that has the following properties:
\begin{enumerate}
\item $G$ has the same finiteness properties as $H$,
\item $H$ is a subgroup of $G$,
\item $G$ admits a quasi-retract onto $H$.
\end{enumerate}
\end{theorem}
\begin{proof}
From Proposition~\ref{prop:realizing-G-as-persistent-retract} we know that there is a self-similar subgroup $\Gamma$ of $\Aut(\TT_d)$ for some $d \in \N$ that has the same finiteness properties as $H$ and that admits a persistent retraction $r \colon \Gamma \rightarrow H$.
In this case, Corollary~\ref{cor:from-pers-retr-to-virt-simpl-and-qr} provides us with an embedding $j \colon \Gamma \rightarrow \Aut(\TT_{m})$ for some $m \in \N$ such that $j(\Gamma)$ is a self-similar subgroup of $\Aut(\TT_{m})$ whose corresponding R\"over--Nekrashevych group $V_{m}(j(\Gamma))$ has finite abelianization and admits a quasi-retraction $\rho \colon V_{m}(j(\Gamma)) \rightarrow H$.
According to Proposition~\ref{prop:positive-finiteness-properties} (resp.~\cite[Section 4]{SkipperWitzelZaremsky19}) the group $V_{m}(j(\Gamma))$ is of type $\FP_n(R)$ for some unital commutative ring $R$ (resp.\ of type $\F_n$) if $j(\Gamma) \cong \Gamma$ is of type $\FP_n(R)$ (resp.\ of type $\F_n$).
On the other hand, Theorem~\ref{thm-quasi-re-fin} tells us that $V_{m}(j(\Gamma))$ is of type $\FP_n(R)$ (resp.\ of type $\F_n$) only if $H$ is of type $\FP_n(R)$ (resp.\ of type $\F_n$).
Since $\Gamma$ has the same finiteness properties as $H$, we can summarize the statements by saying that $V_{m}(j(\Gamma))$ has the same finiteness properties as $H$.

As $V_{m}(j(\Gamma))$ has finite abelianization, it follows from Proposition~\ref{prop:prpty-Vd(G)} that its commutator subgroup $G \defeq [V_{m}(j(\Gamma)),V_{m}(j(\Gamma))]$ is a simple subgroup of finite index.
Thus the inclusion map $i \colon G \rightarrow V_{m}(j(\Gamma))$ is a quasi-isometry so that another application of Theorem~\ref{thm-quasi-re-fin} implies that $G$ and $H$ have the same finiteness properties.
Moreover we can precompose $\rho$ with $i$ to obtain a quasi-retraction $\rho \circ i \colon G \rightarrow H$.
It remains to prove that $H$ is a subgroup of $G$.
Since $G$ has finite index in $V_{m}(j(\Gamma))$ and $H$ is a subgroup of $V_{m}(j(\Gamma))$, it follows from the Kaloujnine--Krasner embedding Theorem~\cite{KaloujnineKrasner48} that it is sufficient to prove that the wreath product $S_k \wr G$ for $k\in \N$ embeds into $G$.
Since the latter was proven by Zaremsky~\cite[Proposition 2.5]{Zaremsky25embedding}, this completes the proof.
\end{proof}

Using Theorem~\ref{thm:linear-groups-as-retracts}, we can now easily prove our main result.

\begin{theorem}\label{thm:BB-type-simple-groups}
For every finite graph $\Gamma$ there is a simple group $G_{\Gamma}$ with the following finiteness properties, where $R\neq 0$ is any commutative unital ring:
\begin{enumerate}
\item $G_{\Gamma}$ is of type $\F_n$ if and only if $\Flag(\Gamma)$ is $(n-1)$-connected;
\item $G_{\Gamma}$ is of type $\FP_n(R)$ if and only if $\Flag(\Gamma)$ is homologically $(n-1)$-connected over $R$;
\item $G_{\Gamma}$ is $\F_\infty$ if and only if $\Flag(\Gamma)$ is contractible;
\item $G_{\Gamma}$ is $\FP_\infty(R)$ if and only if $\Flag(\Gamma)$ is $R$-acyclic.
\end{enumerate}
\end{theorem}
\begin{proof}
By the main theorem of~\cite{BestvinaBrady97}, the Bestvina-Brady group $\BB_{\Gamma}$ is of type $\FP_{n}(R)$ (resp.\ finitely presented) if and only if $\Flag(\Gamma)$ is homologically $(n-1)$-connected (resp.\ simply connected).
Thus, by the Hurewicz theorem, $\BB_{\Gamma}$ is of type $\F_{n}$ if and only if $\Flag(\Gamma)$ is $(n-1)$-connected.
Since $\BB_{\Gamma}$ is a subgroup of a right-angled Artin group $A_{\Gamma}$, which in turn is a subgroup of a right-angled Coxeter group $W_{\Gamma}$, it follows from the integrality of the cosine matrix associated to $W_{\Gamma}$ that $\BB_{\Gamma}$ is linear over $\Z$.
In the case where $\BB_{\Gamma}$ is finitely generated, i.e.\ if $\Flag(\Gamma)$ is $0$-connected, we can apply Theorem~\ref{thm:linear-groups-as-retracts} to deduce that there exists a simple group $G$ that is of type $\F_n$ (resp.\ $\FP_n(R)$) if and only if $\BB_{\Gamma}$ is of type $\F_n$ (resp.\ $\FP_n(R)$).
In the case where $\Gamma$ is not $0$-connected, the claim of the Theorem is that there exists a simple group that is not finitely generated, which is well-known to be the case.
\end{proof}

As a byproduct of the proof of Theorem~\ref{thm:BB-type-simple-groups}, we obtain the first known examples of self-similar groups that are separated by homological finiteness properties.

\begin{proposition}\label{prop:BB-type-selfsim-groups}
For every $n \in \N$ there is a self-similar group of type $\FP_{n}(\mathbb{Z})$ that is neither of type $\FP_{n+1}(\mathbb{Z})$ nor finitely presented.
Moreover there is a self-similar group of type $\FP_{\infty}(\mathbb{Z})$ that is not finitely presented.
\end{proposition}
\begin{proof}
From Proposition~\ref{prop:realizing-G-as-persistent-retract} it follows that for every finitely generated subgroup $G$ of $\GL_n(\Q)$ there is a self-similar group that has the same finiteness properties as $G$.
Thus the claim follows from the existence of Bestvina-Brady groups with the desired finiteness properties and the fact that Bestvina-Brady groups are linear over $\Q$, see the proof of Theorem~\ref{thm:BB-type-simple-groups}.
\end{proof}

\begin{corollary}\label{cor:higher-dehn-fct-general}
Every (high-dimensional) Dehn function of a subgroup of $\GL_n(\Q)$ arises as a lower bound on the corresponding Dehn function of a simple group with the same finiteness properties.
\end{corollary}
\begin{proof}
This is a direct consequence of Theorem~\ref{thm:linear-groups-as-retracts} and Theorem~\ref{thm:quasi-retract-dehnfunc}.
\end{proof}

\begin{corollary}\label{cor:higher-dehn-fct}
For every finite subset $I \subseteq \N$, there is a simple group of type $\F_{\infty}$ whose $n$-dimensional Dehn function is at least exponential for every $n \in I$.
\end{corollary}
\begin{proof}
For each $n \in \N$, let $G_n \leq \GL_n(\Q)$ be a subgroup whose $n$-dimensional Dehn function is exponential.
By a result of Leuzinger and Young~\cite{LeuzingerYoung21} examples of such groups are given by $\SL_{n+2}(\Z)$.
Then the product $H \defeq \prod \limits_{n \in I} G_n$ is a subgroup of $\GL_N(\Q)$ for an appropriate $N \in \N$.
From Theorem~\ref{introthm:linear-groups-as-retracts} we know that there is a simple group $G$ with the same finiteness properties as $H$ that admits a quasi-retract onto $H$.
Since each of the groups $\SL_{n+2}(\Z)$ is of type $\F_{\infty}$~\cite{Raghunathan1968} it follows that $G$ is of type $\F_{\infty}$.
By precomposing the canonical projection $H \rightarrow G_n$ with the quasi-retraction $G \rightarrow H$, we obtain a quasi-retract from $G$ onto $G_n$ for each $n \in I$.
In this case Corollary~\ref{cor:higher-dehn-fct-general} tells us that the $n$-dimensional Dehn function of $G$ is bounded from below by the $n$-dimensional Dehn function of $G_n$, which completes the proof.
\end{proof}

\bibliographystyle{alpha}
\bibliography{literature.bib}

\end{document}